\documentclass[11pt, reqno]{amsart}

\evensidemargin
\oddsidemargin
\makeatletter\@addtoreset{equation}{section}\makeatother

\usepackage{amsmath, amsthm, latexsym, amssymb, bbm}
\usepackage[dvips]{color}

\usepackage{calc,  subfigure, bm}
\usepackage[dvips]{graphicx}

\renewcommand{\phi}{\varphi}

\newcommand{\R}{\mathbb R}

\newcommand{\N}{\mathbb N}

\newcommand{\E}{\mathbb E}
\renewcommand{\P}{\mathbb P}

\newtheorem{theorem}{Theorem}[section]

\newtheorem{lemma}[theorem]{Lemma}

\newtheorem{prop}[theorem]{Proposition}

\def\1{{\mathchoice {1\mskip-4mu\mathrm l}      
{1\mskip-4mu\mathrm l}
{1\mskip-4.5mu\mathrm l} {1\mskip-5mu\mathrm l}}}

\renewcommand{\subsection}{\secdef \subsct\sbsect}
\newcommand{\subsct}[2][default]{\refstepcounter{subsection}
\vspace{0.15cm}
{\flushleft\bf \arabic{section}.\arabic{subsection}~\bf #1  }
\nopagebreak\nopagebreak}
\newcommand{\sbsect}[1]{\vspace{0.1cm}\noindent
{\bf #1}\vspace{0.1cm}}

\newcounter{remnr}
\newenvironment{remark}{\refstepcounter{remnr}
{\sf Remark~\arabic{remnr}.\ }\nopagebreak  }%
{\nopagebreak {\hfill{$\diamond$}}\\ }

\renewcommand{\phi}{\varphi}

\renewcommand{\P}{\mathbb{P}}
\renewcommand{\E}{\mathbb{E}}

\begin{document}

\title{Time-dependent P\' olya urn}

\author[Nadia Sidorova]{}

\maketitle

\centerline{\sc Nadia Sidorova\footnote{Department of Mathematics, University College London, Gower Street, London WC1 E6BT, UK, {\tt n.sidorova@ucl.ac.uk}.
} }

\vspace{0.4cm}


\vspace{0.4cm}

\begin{quote}
{\small {\bf Abstract:} We consider a time-dependent version of a P\' olya urn containing black and white balls. At each time $n$ a ball is drawn from the urn at random and replaced in the urn along with $\sigma_n$ 
additional balls of the same colour. The proportion of white balls converges almost surely to a random limit $\Theta$, and $\mathcal{D}=\{\Theta\in\{0,1\}\}$ denotes the event when one of the colours dominates. 
The phase transition, in terms of the sequence $(\sigma_n)$, 
between the regimes $\P(\mathcal{D})=1$ and $\P(\mathcal{D})<1$
was found in~\cite{P}. We describe the phase transition between the cases 
$\P(\mathcal{D})=0$ and $\P(\mathcal{D})>0$. Further, we study the stronger monopoly event 
$\mathcal{M}$ when one of the colours eventually stops reappearing, and   
analyse the phase transition between the regimes $\P(\mathcal{M})=0$, $\P(\mathcal{M})\in (0,1)$, and $\P(\mathcal{M})=1$.}
\end{quote}
\vspace{5ex}

{\small {\bf AMS Subject Classification:} Primary 60G42.
Secondary 60G20, 05D40.

}
\vspace{4ex}



\section{Introduction}
\medskip

%

The Polya urn model was introduced in 1923 by Eggenberger and P\' olya, \cite{EP}, and is a classical example of a random process with reinforcement, \cite{Ps}.  
The model has an urn containing black and white balls.  
At each time $n$ a ball is drawn from the urn at random and replaced in the urn along with another ball of the same colour. 
It is well-known that the proportion 
of the balls of each colour converges almost surely to a random variable, 
which has Beta distribution, \cite{F}. 
\smallskip

We consider a \emph{time-dependent} urn, introduced in~\cite{P}, where the number of balls added at time $n$ is no longer one but 
is a function of $n$. 
Let $(\sigma_n)$ be a positive sequence representing the number of added balls. 
Denote by $\tau_0$ the initial number of balls in the urn and, for each $n$, let 
\begin{align*}
\tau_n=\tau_0+\sum_{i=1}^n\sigma_i
\end{align*}
be the total number of balls at time $n$. 

Denote by $T_0$ the initial deterministic number of white balls.
Given that the urn contains $T_{n}$ white balls at time $n$, we define 
\begin{align}
\label{def_t}
T_{n+1}=T_{n}+\sigma_{n+1} I_{n+1},
\end{align}  
where $I_{n+1}$ is a Bernoulli random variable with parameter 
\begin{align}
\label{def_th}
\Theta_n=\frac{T_n}{\tau_n},
\end{align}
otherwise independent of $\mathcal{F}_n=\sigma(I_1,\dots,I_n)$. 
\smallskip

In the urn context, it is natural to think of $(\sigma_n)$ as an integer-valued sequence. However, 
the random processes $(T_n)$ and $(\Theta_n)$ are well-defined for real-valued positive $(\sigma_n)$. 
In the sequel we allow $(\sigma_n)$ to take non-integer values
and use~\eqref{def_t} and~\eqref{def_th} as the definitions of the processes $(T_n)$ and $(\Theta_n)$, respectively.
\smallskip

It follows from the same martingale argument as for the standard P\' olya urn that the proportion $\Theta_n$ of white balls  
converges almost surely to a random variable $\Theta$. However, little is known about the distribution of $\Theta$. It was shown in~\cite{P} that 
$\P(\Theta=0)+\P(\Theta=1)=1$ if and only if 
\begin{align}
\label{serser}
\sum_{n=0}^{\infty}\Big(\frac{\sigma_{n+1}}{\tau_n}\Big)^2=\infty,
\end{align}
that is, if and only if $(\sigma_n)$ grows sufficiently fast. 
However, the regime when  
\begin{align}
\label{serser1}
\sum_{n=0}^{\infty}\Big(\frac{\sigma_{n+1}}{\tau_n}\Big)^2<\infty
\end{align}
has not been well understood. 
It is only known that  $\P(\Theta=0)=\P(\Theta=1)=0$ if $(\sigma_n)$ is bounded, see~\cite{P}, 
and, obviously, if $(\sigma_n)$ is decaying so fast that $(\tau_n)$ converges. 
The only other result available  for this regime, again proved in~\cite{P},  is that 
$\Theta$ has no atoms in $(0,1)$. 
\smallskip

We denote by 
\begin{align*}
\mathcal{D}=\big\{\Theta=0\big\}\cup \big\{\Theta=1\big\}
\end{align*}
the event that eventually the number of balls of one colour is  negligible with respect to the number of balls of the other colour,
and call this event~\emph{domination}. Further, we denote by 
\begin{align*}
\mathcal{M}=\big\{I_n=0\text{ eventually for all }n\big\}\cup \big\{I_n=1\text{ eventually for all }n\big\}
\end{align*}
the event that eventually only balls of one colour are added to the urn, and call this event~\emph{monopoly}. 
We have 
\begin{align*}
\mathcal{M}\subset \mathcal{D}.
\end{align*} 
This is obvious if $\tau_n\to\infty$ and will be 
shown in Lemma~\ref{l:ml} for the case when $\tau_n$ converges. 
\smallskip

As discussed above, in this terminology~\eqref{serser} implies $\P(\mathcal{D})=1$ while~\eqref{serser1} implies $\P(\mathcal{D})<1$. The aim of the paper is to find a phase transition between the cases $\P(\mathcal{D})>0$ and $\P(\mathcal{D})=0$ in the regime~\eqref{serser1}.  
This phase transition is closely related to that between
$\P(\mathcal{M})>0$ and $\P(\mathcal{M})=0$, which we also describe.  

\begin{theorem} 
\label{t:con}
Suppose 
\begin{align}
\label{new5}
\sum_{n=1}^{\infty}\frac{1}{\tau_n}<\infty.
\end{align}
Then $0<\P(\mathcal{M})\le \P(\mathcal{D})$. 
\smallskip

If we additionally assume
\begin{align}
\label{new6}
\liminf_{n\to\infty}\frac{\sigma_n}{\tau_n}>0
\end{align}
then $\P(\mathcal{M})=\P(\mathcal{D})=1$.
%
%
\end{theorem}

\begin{remark} 
Condition~\eqref{new5} is satisfied by sequences $(\sigma_n)$ growing like $(\log n)^{\alpha}$, $\alpha>1$, or faster. Theorem~\ref{t:con} means that for such sequences we have monopoly and hence domination occurring with positive probability. 
In particular, for such sequences growing slower than $e^{\sqrt n}$ 
condition~\eqref{serser1} is satisfied and we have $0<\P(\mathcal{M})\le \P(\mathcal{D})<1$.
On the other hand, for sequences growing like
$e^{\sqrt n}$ or faster condition~\eqref{serser} holds and we have
$0<\P(\mathcal{M})\le \P(\mathcal{D})=1$.
\end{remark}

\begin{remark}
Condition~\eqref{new6} is satisfied by sequences $(\sigma_n)$ growing exponentially or faster. Such sequences satisfy~\eqref{serser}, which immediately implies $\P(\mathcal{D})=1$.  The aim of the second statement of the theorem 
is to show that for fast-growing sequences not just domination but also monopoly  
occurs almost surely.  
\end{remark}

\begin{remark} We do not know anything about the event $\mathcal{D}\backslash \mathcal{M}$. Understanding this event is equivalent to understanding the event
\begin{align*}
\big\{\Theta= 0\big\}\backslash 
\Big\{\sum_{n=1}^{\infty}\Theta_n<\infty\Big\}.
\end{align*}
This requires the knowledge about the rate of convergence of $\Theta_n$ to zero, which is currently beyond our ken. 
\end{remark}

For two positive sequences $(a_n)$ and $(b_n)$, we say that 
$a_n\asymp b_n$ as $n\to\infty$ if the sequence $\big(\frac{a_n}{b_n}\big)$ is bounded away from zero and infinity.

\begin{theorem} 
\label{t:div}
Suppose
\begin{align}
\label{new2}
\sum_{n=1}^{\infty}\frac{1}{\tau_n}=\infty.
\end{align}
Then $\P(\mathcal{M})=0$.
\smallskip

Suppose that additionally the following regularity conditions are satisfied:
\begin{itemize}
\item[(1)] As $n\to\infty$,
\begin{align}
\label{rc1}
\frac{\sigma_n}{\tau_n}\asymp \frac 1 n.
\end{align}
\item[(2)] There exists $g:\N\to \N$ satisfying $g(n)\to\infty$ as $n\to\infty$ and $\alpha,\beta>0$ such that 
\begin{align}
\label{rc2}
\alpha<\frac{\sigma_i}{\sigma_n}<\beta
\end{align}
for all $n$ and all $n\le i\le ng(n)$.
\end{itemize}
Then
$\P(\mathcal{D})=0$.
\end{theorem}

\begin{remark} Condition~\eqref{new2} is satisfied by sequences $(\sigma_n)$ growing like $\log n$ or slower, or decaying. 
For growing sequences $(\sigma_n)$ satisfying~\eqref{new2},  the conditions \eqref{rc1} and~\eqref{rc2} are always 
fulfilled provided that the sequence is regular enough. 
They are  also  fulfilled for all bounded sequences. 
Overall, this theorem works for sufficiently regular sequences growing like $\log n$ or slower.
\end{remark}

\begin{remark}
In principle, Theorem~\ref{t:div} also works for sequences 
decaying like $n^{-\alpha}$, $\alpha<1$, or slower. 
However, we have a much simpler Theorem~\ref{t:zero} for the case when $\sigma_n\to 0$. 
\end{remark}

\begin{theorem} 
\label{t:zero}
If $\sigma_n\to 0$ as $n\to\infty$ then
$\P(\mathcal{M})=\P(\mathcal{D})=0$.
\end{theorem}
\smallskip

Another time-inhomogeneous model of a Polya urn was studied 
in~\cite{davis}. In that model
$\sigma_n$ white balls were added to the urn the $n$-th time a white ball was drawn, and $\hat \sigma_n$ black balls were added the $n$-th time a black ball was drawn. That model, however, exhibits a very different behaviour and can be treated using Rubin's exponential embedding, which is not applicable to our model.
\smallskip

The paper is organised as follows.
\smallskip

We begin by proving Theorem~\ref{t:con}. 
For the first statement we simply show that the probability of never adding a black ball is positive if~\eqref{new5} holds. To justify the second statement we observe that if no monopoly occurs then white balls will be added to the urn infinitely often. At each time $n$ when white balls are added their proportion $\Theta_n$ becomes at least $\frac{\sigma_n}{\tau_n}$ and hence cannot converge to zero according to~\eqref{new6}. By a symmetric argument it cannot converge to one either, implying $\mathcal{M}=\mathcal{D}$. It remains to notice that $\P(\mathcal{D})=1$ since~\eqref{new6} implies~\eqref{serser}.    
\smallskip

Further, for each $n\in\N_0$, we denote 
\begin{align*}
\delta_n=\sum_{i=n+1}^{\infty}\Big(\frac{\sigma_i}{\tau_i}\Big)^2\in (0,\infty].
\end{align*}
In key Proposition~\ref{p:delta} we show that if $\Theta_n$ decays slower 
than~$\delta_n$ then it actually does not decay at all. 
\smallskip

Then we turn to the proof of Theorem~\ref{t:div}. The absence of monopoly follows from~\eqref{new2} by a Borel-Cantelli type argument. Further, if no monopoly occurs then there will be infinitely many times $n$ when white balls are added to the urn. At each such time $\Theta_n$ becomes at least 
$\frac{\sigma_n}{\tau_n}\asymp \frac 1 n$. This remains true for all $n\le i<ng(n)$, that is, $\Theta_i>\frac{\lambda}{i}$, with some 
$\lambda>0$. This implies that the average number of times, between $n$ and $ng(n)$, when white balls were added to the urn is greater than 
\begin{align*}
\lambda\sum_{i=n}^{ng(n)-1}\frac{1}{i}\ge \lambda\int_n^{ng(n)}\frac{dx}{x}
=\lambda\log g(n).
\end{align*} 
This enables us to show that $\Theta_{ng(n)}$ decays slower than 
$\delta_{ng(n)}$ along a subsequence, and the rest of the proof follows from Proposition~\ref{p:delta}.
\smallskip

We conclude the paper by proving Theorem~\ref{t:zero}. We show that $\delta_n$ decays faster than $\frac{1}{\tau_n}$ and $\Theta_n$ decays slower than $\frac{1}{\tau_n}$. The rest follows from Proposition~\ref{p:delta}.

\bigskip


\section{Proofs}

\begin{lemma} 
\label{l:ml}
$\mathcal{M}\subset\mathcal{D}$.
\end{lemma}

\begin{proof} Observe that if $\tau_n\to\infty$ the statement is obvious. For the general case 
by L\' evy's extension of the Borel-Cantelli Lemmas, see~\cite[\S 12.15]{DW}, we have 
\begin{align*}
\big\{I_n=0\text{ eventually for all }n\big\}
=\Big\{\sum_{n=1}^{\infty}\E_{\mathcal{F}_n}I_{n+1}<\infty\Big\}
=\Big\{\sum_{n=1}^{\infty}\Theta_n<\infty\Big\}\subset \big\{\Theta=0\big\},
\end{align*}
as required.
\end{proof}
\bigskip

\begin{proof}[Proof of Theorem~\ref{t:con}]
By symmetry and by Lemma~\ref{l:ml} it suffices to prove that 
\begin{align*}
\P(I_n=0\text{ eventually for all }n)>0.
\end{align*}
We have by~\eqref{new5}
\begin{align*}
\P(I_n=0\text{ eventually for all }n)
&\ge \P(I_n = 0\text{ for all }n)
=\prod_{n=0}^{\infty}\P(I_{n+1}=0|I_1=0,\cdots,I_{n}=0)\\
&=\prod_{n=0}^{\infty}\Big(1-\frac{T_0}{\tau_n}\Big)
=\exp\Big\{\sum_{n=0}^{\infty}\log\Big(1-\frac{T_0}{\tau_n}\Big)\Big\}>0
\end{align*}
since $\Theta_{n}=\frac{T_0}{\tau_n}$ on the event $\{I_1=0,\cdots,I_{n}=0\}$.
\smallskip

Suppose condition~\eqref{new6} is satisfied. Then 
\begin{align*}
\liminf_{n\to\infty}\frac{\sigma_{n+1}}{\tau_n}\ge \liminf_{n\to\infty}\frac{\sigma_{n+1}}{\tau_{n+1}}>0
\end{align*}
implying~\eqref{serser} and $\P(\mathcal{D})=1$.
\smallskip

For all $\omega\in\mathcal{M}^c$ there is a random sequence of indices $(\eta_n)$ such that 
$I_{\eta_n}=1$. Then 
\begin{align*}
\Theta_{\eta_n}=\frac{T_{\eta_n-1}+\sigma_{\eta_n}}{\tau_{\eta_n}}\ge \frac{\sigma_{\eta_n}}{\tau_{\eta_n}}.
\end{align*}
This implies 
\begin{align*}
\Theta\ge \liminf_{n\to\infty}\frac{\sigma_n}{\tau_n}>0.
\end{align*}
By a symmetric argument we also have $\Theta<1$. 
Hence $\mathcal{M}^c\subset \mathcal{D}^c$. Together with Lemma~\ref{l:ml} this implies $\mathcal{M}=\mathcal{D}$
and $\P(\mathcal{M})=\P(\mathcal{D})=1$.
\end{proof}
\bigskip

\begin{lemma} 
\label{bounded}
There exists $c>0$ such that 
\begin{align*}
1-p+pe^{-x}\le e^{-px+cpx^2}
\end{align*}
for all $p\in (0,1)$ and all $x> 0$. 
\end{lemma}

\begin{proof} 
Using $\log(1+x)\le x$ we have 
\begin{align*}
\sup_{x> 0, p\in (0,1)}\frac{x+p^{-1}\log(1-p+pe^{-x})}{x^2}
\le \sup_{x> 0}\frac{x-1+e^{-x}}{x^2}<\infty,
\end{align*}
as the function under the supremum is continuous and tends to finite limits at zero and  infinity. It is easy to see that this is equivalent to the statement of the lemma.
\end{proof}
\bigskip

\begin{prop} 
\label{p:delta}
Suppose
\begin{align}
\label{mas}
\limsup_{n\to\infty}\frac{\Theta_n}{\delta_n}= \infty
\end{align}
almost surely. Then $\P(\Theta=0)=0$. 
\end{prop}

\begin{proof} 
Observe that~\eqref{mas} implies that all $\delta_n$ are finite. 
Denote by 
\begin{align*}
f_n(\lambda)=\E e^{-\lambda \Theta_n}
\qquad\text{and}\qquad 
f(\lambda)=\E e^{-\lambda \Theta},
\qquad \lambda\in\R
\end{align*} 
the Laplace transforms of $\Theta_n$ and $\Theta$, respectively. 
For all $\lambda>0$, all $m$, and all $n>m$ 
using 
\begin{align*}
\Theta_n=\frac{T_{n-1}+\sigma_n I_n}{\tau_n}=\frac{\tau_{n-1}}{\tau_n}\Theta_{n-1}+\frac{\sigma_n}{\tau_n}I_n
\end{align*}
we have  
\begin{align*}
f_n(\lambda)
&=\E\Big[\exp\Big\{-\frac{\tau_{n-1}}{\tau_n}\lambda\Theta_{n-1}\Big\}\Big(1-\Theta_{n-1}+\Theta_{n-1}\exp\Big\{-\lambda\frac{\sigma_n}{\tau_n}\Big\}\Big)\Big].
\end{align*}
By Lemma~\ref{bounded} this implies 
\begin{align}
f_n(\lambda)
&\le\E\Big[\exp\Big\{-\lambda\Theta_{n-1}
+c\Big(\frac{\sigma_n}{\tau_n}\Big)^2\lambda^2\Theta_{n-1}
\Big\}\Big]
= f_{n-1}\Big(\lambda-c\Big(\frac{\sigma_n}{\tau_n}\Big)^2\lambda^2\Big).
\label{it7}
\end{align}

Let 
\begin{align*}
\lambda_m=\frac{1}{2c\delta_m}.
\end{align*}

Let us prove by induction over $k$ that  
\begin{align}
\label{indu}
f_n(\lambda_m)\le f_{n-k}\Big(\lambda_m-c\lambda_m^2\sum_{i=n-k+1}^{n}\Big(\frac{\sigma_i}{\tau_i}\Big)^2\Big)
\end{align}
for all $m$, all $n>m$ and all $1\le k\le n-m$. Indeed, for $k=1$ the statement follows from~\eqref{it7}. 
Further, suppose it is true for some $k$. 
Observe that 
\begin{align}
\label{bbb}
\lambda_m-c\lambda_m^2\sum_{i=n-k+1}^{n}\Big(\frac{\sigma_i}{\tau_i}\Big)^2
\ge \lambda_m-c\lambda_m^2\sum_{i=m+1}^{\infty}\Big(\frac{\sigma_i}{\tau_i}\Big)^2
=\lambda_m\big(1-c\lambda_m\delta_m\big)
=\frac{\lambda_m}{2}>0.
\end{align}
Hence using~\eqref{it7} and monotonicity of $f_{n-k-1}$ we obtain 
\begin{align*}
 f_{n-k}&\Big(\lambda_m-c\lambda_m^2\sum_{i=n-k+1}^{n}\Big(\frac{\sigma_i}{\tau_i}\Big)^2\Big)\\
& \le f_{n-k-1}\Big(\lambda_m-c\lambda_m^2\sum_{i=n-k+1}^{n}\Big(\frac{\sigma_i}{\tau_i}\Big)^2
 -c\Big(\lambda_m-c\lambda_m^2\sum_{i=n-k+1}^{n}\Big(\frac{\sigma_i}{\tau_i}\Big)^2\Big)^2\Big(\frac{\sigma_{n-k}}{\tau_{n-k}}\Big)^2\Big)\\
& \le f_{n-k-1}\Big(\lambda_m-c\lambda_m^2\sum_{i=n-k}^{n}\Big(\frac{\sigma_i}{\tau_i}\Big)^2\Big).
\end{align*}
This, together with the induction hypothesis~\eqref{indu} for $k$, completes the induction step.  
\smallskip

Substituting $k=n-m$ into~\eqref{indu} and using~\eqref{bbb} we obtain
\begin{align}
f_n(\lambda_m)
&\le f_{m}\Big(\lambda_m-c\lambda_m^2\sum_{i=m+1}^{n}\Big(\frac{\sigma_i}{\tau_i}\Big)^2\Big)
\le f_m \Big(\frac{\lambda_m}{2}\Big)
=\E \exp\Big\{-\frac{\Theta_m}{4c\delta_m}\Big\}
\label{lili7}
\end{align}
for all $m$ and all $n>m$. 
\smallskip

By the dominated convergence theorem we have  $f_n(\lambda)\to f(\lambda)$ 
for all $\lambda>0$.  Hence we can take the limit in~\eqref{lili7} to obtain 
\begin{align*}
\P(\Theta=0)\le
f(\lambda_m)\le \E \exp\Big\{-\frac{\Theta_m}{4c\delta_m}\Big\}
\end{align*}
for all $m$. 
By~\eqref{mas} there is a subsequence $(m_i)$ such that 
\begin{align*}
\lim_{i\to\infty}\frac{\Theta_{m_i}}{\delta_{m_i}}=\infty. 
\end{align*}
By the dominated convergence theorem we obtain 
\begin{align*}
\P(\Theta=0)\le\lim_{i\to\infty}\E \exp\Big\{-\frac{\Theta_{m_i}}{4c\delta_{m_i}}\Big\}=0
\end{align*}
as required. 
\end{proof} 
\bigskip

\begin{proof}[Proof of Theorem~\ref{t:div}] Suppose~\eqref{new2} is satisfied. Let us show that $\P(\mathcal{M})=0$. 
By symmetry it suffices to prove that 
\begin{align*}
\P(I_n=0\text{ eventually for all }n)=0.
\end{align*}
By L\' evy's extension of the Borel-Cantelli Lemmas, see~\cite[\S 12.15]{DW}, 
this is equivalent to 
\begin{align*}
\sum_{n=1}^{\infty}\E_{\mathcal{F}_{n}}I_{n+1}=\sum_{n=1}^{\infty}\Theta_n=\infty
\end{align*}
almost surely, which follows from 
\begin{align*}
\Theta_n=\frac{T_n}{\tau_n}\ge \frac{1}{\tau_n}
\end{align*}
and~\eqref{new2}.
\bigskip

Now suppose~\eqref{new2} and the regularity conditions~\eqref{rc1} and~\eqref{rc2} are satisfied.
Let us show that $\P(\mathcal{D})=0$.
By symmetry it suffices to prove that $\P(\Theta=0)=0$. To do so we will use Proposition~\ref{p:delta}.
\smallskip

It follows from~\eqref{rc1} that there are constants $a,b>0$ such that for all $n$
\begin{align*}
\frac a n<\frac{\sigma_n}{\tau_n}<\frac b n.
\end{align*}
This implies 
\begin{align}
\label{dede}
\delta_n=\sum_{i=n+1}^{\infty}\Big(\frac{\sigma_i}{\tau_i}\Big)^2\le b^2\sum_{i=n+1}^{\infty}\frac{1}{i^2}\le b^2\int_n^{\infty}\frac{dx}{x^2}=\frac{b^2}{n}.
\end{align}

For all $n$, denote $f(n)=ng(n)$. Let $\gamma_0=0$ and for all $n\in\N_0$
\begin{align*}
\gamma_{n+1}=\inf\big\{i>f(\gamma_{n}): I_i=1\big\}.
\end{align*}
By the first part of the theorem we know that  all $\gamma_n$ are finite almost surely. 
\smallskip

Observe that for all $n\in\N$
\begin{align}
\label{io2}
\Theta_{\gamma_n}= \frac{T_{\gamma_n-1}+\sigma_{\gamma_n}}{\tau_{\gamma_n}}\ge \frac{\sigma_{\gamma_n}}{\tau_{\gamma_n}}
>\frac{a}{\gamma_n}.
\end{align}

Further, for all $\gamma_n\le i< f(\gamma_n)$ it follows from~\eqref{io2} and~\eqref{rc2}  that 
\begin{align}
\Theta_i=\frac{T_i}{\tau_i}\ge \frac{T_{\gamma_n}}{\tau_i}\ge \Theta_{\gamma_n}\frac{\tau_{\gamma_n}}{\tau_i}
\ge  \frac{a}{\gamma_n}\cdot \frac{\tau_{\gamma_n}}{\sigma_{\gamma_n}}\cdot \frac{\sigma_i}{\tau_i}\cdot \frac{\sigma_{\gamma_n}}{\sigma_i}
\ge \frac{a}{\gamma_n}\cdot \frac{\gamma_n}{b}\cdot \frac{a}{i}\cdot \frac{1}{\beta}
= \frac{\lambda}{i},
\label{2i}
\end{align}
where $\lambda=\frac{a^2}{b\beta}$.
Denote 
\begin{align*}
Z_n=\sum_{i=\gamma_n+1}^{f(\gamma_n)}I_i.
\end{align*}
Let 
\begin{align*}
\hat Z_n=\sum_{i=\gamma_n+1}^{f(\gamma_n)}\hat I_{i},
\end{align*}
where $(\hat I_{i})$ is a sequence of independent Bernoulli 
random variables with parameters $\frac{\lambda}{i}$, respectively,
coupled with $(I_i)$ in such a way that 
\begin{align*}
I_i\ge \hat I_{i} \qquad \text{for all } \gamma_n<i\le f(\gamma_n).
\end{align*}
Observe that such a coupling is possible by~\eqref{2i} and it implies  that 
\begin{align}
\label{zz}
Z_n\ge \hat Z_n
\end{align}
for all $n$ almost surely. 
We have 
\begin{align*}
\E_{\mathcal{F}_{\gamma_n}}\hat Z_n
= \sum_{i=\gamma_n+1}^{f(\gamma_n)}\frac{\lambda}{i}
=\lambda \log g(\gamma_n)+o(1)
\end{align*}
and 
\begin{align*}
\text{Var}_{\mathcal{F}_{\gamma_n}} \hat Z_n
=  \sum_{i=\gamma_n+1}^{f(\gamma_n)}\frac{\lambda}{i}
=\lambda \log g(\gamma_n)+o(1)
\end{align*}
as $n\to\infty$, and the convergences are uniform in $\omega$ since $\gamma_n\ge n$ for all $n$ almost surely. By Markov's inequality and using~\eqref{zz} we have 
\begin{align*}
\P_{\mathcal{F}_{\gamma_n}}\Big(Z_n\le\frac \lambda 3 \log g(\gamma_n)\Big)
&\le \P_{\mathcal{F}_{\gamma_n}}\Big(\hat Z_n\le\frac \lambda 3 \log g(\gamma_n)\Big)
\le \P_{\mathcal{F}_{\gamma_n}}\Big(\Big|\hat Z_n-\E_{\mathcal{F}_{\gamma_n}}\hat Z_n\Big|\ge\frac \lambda 3 \log g(\gamma_n)\Big)\\
&\le \frac{9}{\lambda^2(\log g(\gamma_n))^2}\text{Var}_{\mathcal{F}_{\gamma_n}} \hat Z_n\le\frac{10}{\lambda\log g(\gamma_n)}<\frac 1 2
\end{align*}
eventually for all $n$ uniformly in $\omega$. This implies that 
\begin{align*}
\sum_{n=1}^{\infty}
\P_{\mathcal{F}_{\gamma_n}}\Big(Z_n>\frac \lambda 3 \log g(\gamma_n)\Big)\ge \sum_{n=1}^{\infty} \frac 1 2 =\infty.
\end{align*}
By L\' evy's extension of the Borel-Cantelli Lemmas, see~\cite[\S 12.15]{DW}, 
we obtain that almost surely 
\begin{align}
\label{io}
Z_n> \frac \lambda 3 \log g(\gamma_n)\quad \text{infinitely often.}
\end{align}
For all $n$ satisfying~\eqref{io} we have using~\eqref{rc1} and~\eqref{rc2}
\begin{align*}
f(\gamma_n) \Theta_{f(\gamma_n)}
&=\frac{f(\gamma_n)}{\tau_{f(\gamma_n)}}T_{f(\gamma_n)}\
\ge \frac{f(\gamma_n)}{\tau_{f(\gamma_n)}}\cdot \frac \lambda 3 \log g(\gamma_n)
\cdot \!\!\!\!\!\!\min_{\gamma_n<i\le f(\gamma_n)}\!\!\!\!\!\!\sigma_i\\
&\ge \frac{\sigma_{f(\gamma_n)}}{\tau_{f(\gamma_n)}}\cdot f(\gamma_n)\cdot \frac{\alpha\lambda}{3} \log g(\gamma_n)\cdot \frac{\sigma_{\gamma_n}}{\sigma_{f(\gamma_n)}}
\ge \frac{a\alpha\lambda}{3\beta} \log g(\gamma_n).
\end{align*} 
Together with~\eqref{io} and since $g(n)\to\infty$ this means that 
\begin{align}
\label{io1}
\limsup_{n\to\infty} n\Theta_n=\infty.
\end{align}

Combined with~\eqref{dede} this implies that 
\begin{align*}
\limsup_{n\to\infty}\frac{\Theta_n}{\delta_n}=\infty,
\end{align*} 
and $\P(\Theta=0)=0$ follows from Proposition~\ref{p:delta}.
\end{proof}

%
\bigskip

\begin{proof}[Proof of Theorem~\ref{t:zero}] 
By symmetry it suffices to prove that $\P(\Theta=0)=0$. If $(\tau_n)$ converges this is obvious. Suppose $\tau_n\to\infty$
as $n\to\infty$. 
\smallskip

Observe that, for all $n$,
\begin{align*}
\delta_n\le \sup_{i>n}\sigma_i \sum_{i=n+1}^{\infty}\frac{\sigma_i}{\tau_i^2}
\le \sup_{i>n}\sigma_i\int_{\tau_n}^{\infty}\frac{dx}{x^2}=\frac{1}{\tau_n}\sup_{i>n}\sigma_i.
\end{align*}
Further, for all $n$ almost surely we have 
\begin{align*}
\Theta_n=\frac{T_n}{\tau_n}\ge \frac{1}{\tau_n}.
\end{align*}
This implies 
\begin{align*}
\lim_{n\to\infty}\frac{\Theta_n}{\delta_n}\ge \lim_{n\to\infty}\big[ \sup_{i>n}\sigma_i\big]^{-1}=\infty.
\end{align*}
$\P(\Theta=0)=0$ follows now from Proposition~\ref{p:delta}.
\end{proof}
\bigskip


\end{document}